\documentclass[12pt]{paper}
 \usepackage{geometry}
\usepackage{hyperref}
                \usepackage{amsthm,amsfonts,amsmath,amscd,amssymb,epsfig,verbatim,
}

\usepackage{graphicx}
\usepackage{epstopdf}
 {\title{Weinstein manifolds revisited} 
 \author{Yakov Eliashberg\thanks{Partially supported by NSF grant  DMS-1505910} }

\date{}  
 \setlength{\marginparwidth}{1.2in}
\let\oldmarginpar\marginpar
\renewcommand\marginpar[1]{\-\oldmarginpar[\raggedleft\footnotesize #1]%
{\raggedright\footnotesize #1}}

\parindent=0pt
\parskip=4pt
\hyphenation{ma-ni-fold ma-ni-folds sub-ma-ni-fold sub-ma-ni-folds}
\theoremstyle{plain}
\newtheorem{theorem}{Theorem}[section]
\newtheorem{thm}[theorem]{Theorem}
\newtheorem{corollary}[theorem]{Corollary}

\newtheorem{proposition}[theorem]{Proposition}
\newtheorem{prop}[theorem]{Proposition}
\newtheorem{lemma}[theorem]{Lemma}

\newtheorem{problem}[theorem]{Problem} 
\theoremstyle{remark}

\newtheorem{remark}[theorem]{Remark}
\newtheorem*{remark*}{Remark}
 \newtheorem{Example}[theorem]{Example}
\newtheorem*{example*}{Example}

\theoremstyle{definition}

\newcommand{\wt}{\widetilde}
\newcommand{\wh}{\widehat}

\newcommand{\p}{\partial}

\newcommand{\om}{\omega}

\newcommand{\eps}{\varepsilon}

\newcommand{\Core}{\rm Core}

\newcommand{\Z}{{\mathbb{Z}}}
\newcommand{\R}{{\mathbb{R}}}
\newcommand{\C}{{\mathbb{C}}}

\newcommand{\st}{{\rm st}}

\newcommand{\sign}{{\mathrm sign}}

\newcommand{\const}{{\rm const}}

\newcommand{\Int}{{\rm Int\,}} 

\newcommand{\Id}{\mathrm {Id}}

\newcommand{\So}{\mathrm{Soul}}

\newcommand{\FF}{\mathcal{F}}

\newcommand{\fW}{{\mathfrak W}}
\newcommand{\fL}{{\mathfrak L}}

\numberwithin{figure}{section}

\begin{document}

\maketitle
  \leftline{\small   Department of Mathematics, Stanford University}  


\begin{abstract}
This is a very biased and incomplete survey of some basic notions, old and new results, as well as  open problems  concerning Weinstein symplectic manifolds.
\end{abstract}

\section{Weinstein manifolds, domains, cobordisms}
We begin with a notion of a Liouville domain.
Let $(X,\om)$ be a $2n$-dimensional  compact symplectic manifold with boundary equipped  with an exact symplectic form $\om$. A Liouville structure on $(X,\om)$ is a choice of a primitive $\lambda$, $d\lambda=\om$, called {\em Liouville form} such that $\lambda|_{\p   X}$ is a contact form and the orientation of $\p  X$ by the form $\lambda\wedge d\lambda^{n-1}|_{\p X}$ coincides with its orientation as the boundary of symplectic manifold $(X,\om)$. The vector field $Z$, that is $\om$-dual to $\lambda$, i.e. $\iota(Z)\om= \lambda$, is also called  {\em Liouville}. It satisfies the condition $L_Z\om=\om$ which means that its flow is conformally symplectically expanding. The contact boundary condition is equivalent to the  outward transversality of $Z$ to $\p   X$.  A Liouville domain $ X$ can always be  completed to a {\em Liouville manifold} $\wh X$ by attaching a  cylindrical end:
$$\wh X:= X\cup ({\p  X}\times[0,\infty))$$ and extending $\lambda$ to $\wh X$ as equal to $e^s(\lambda|_{\p X})$ on the attached end. We will be constantly going back and forth between these two tightly related notions of Liouville  domains and Liouville manifolds.

Given a Liouville  structure $ \fL=(X,\om,Z)$ we say that  a Liouville  structure
$\fL'=(X',\om, Z)$ is obtained by  a {\em radial deformation} from $\fL$ if there exists a   function $h:\wh X\to\R$ such that $X'\subset \wh X$ is the image of $X$ under the time $1$ map $\psi:\wh X\to\wh X$ of the flow of the vector field $hZ$ on the completion $\wh X$. The completions of the  radially equivalent Liouville domains $\fL'$ and  $\fL$ are canonically isomorphic.

The space of Liouville structures for $ (X, \om)$ is convex, and hence any two Liouville structures are canonically homotopic. Given a homotopy of  completed  Liouville structures $(\wh X, \om_t, \lambda_t)$ there exists an isotopy $\phi_t:\wh X\to \wh X$ such that $\phi_t^*\om_t=\om_0$. Moreover, one can always arrange that $\phi_t^*\lambda_t=\lambda_0+dH_t$, see \cite{CE12}, Sections 11.1 and 11.2.
In particular on completed Liouville manifolds                       it is always sufficient to consider homotopies fixing the symplectic form, and, moreover, changing the Liouville form by adding an exact form. Homotopic non-completed Liouville domains  are symplectomorphic {\em up to   radial deformation}.

Given a Liouville domain $\fL=(X,\om,\lambda)$ consider a compact set 
$$\Core(\fL)=\mathop{\bigcap}\limits_{t>0}Z^{-t}( X),$$ the attractor of the negative  flow of
the Liouville vector field $Z$. We will call  $\Core(\fL)$ the {\em core}, or the {\em skeleton} of the Liouville structure $\fL$. While $\Core(\fL)$  has obviously its $2n$-dimensional Lebesgue  measure equal to $0$, it still can be pretty large if no  extra conditions are imposed on the Liouville structure. For instance, McDuff constructed in  \cite{McD-FE}   a Liouville structure on $T^*S_g\setminus S_g$ for a closed surface $S_g$  of genus $g>1$, whose core has   codimension 1. 

However, the situation changes if one requires existence of a Lyapunov function for the Liouville vector field $Z$.
A {\em Weinstein structure} on a domain $ X$ is a Liouville structure $\fL$ together with a   function $\phi: X\to \R$ which is {\em Lyapunov} for the Liouville vector field $Z$, i.e.
 
 \begin{description}
\item{(L1)} $d\phi(Z)>c||Z||^2$ for a positive constant $c$ and some Riemann metric on $  X$.
\end{description}

\begin{figure}[h]
\begin{center}
\includegraphics[scale=.9]{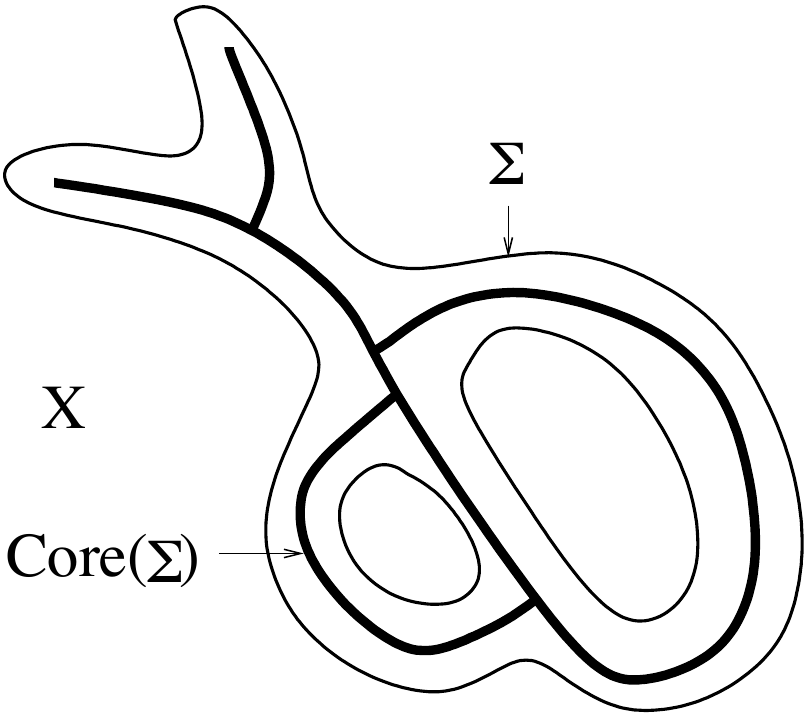}
\caption{Skeleton of a Weinstein domain}\label{fig:skel}
\end{center}  
\end{figure}

 Note that condition (L1) implies that $\Core( X,\lambda)$ is the    union of $Z$-stable  manifolds of critical points of $\phi$ (i.e. points converging to the critical locus in forward time). 
In \cite{CE12}  it was  required   in addition that $\phi$ is either   Morse or generalized Morse (i.e. may have   death-birth critical points).
 Under these assumptions  it was shown in  \cite{CE12}, see also \cite{EliGro91, Eli95},  that
  \begin{itemize}
\item[(L2)] the core is  stratified by isotropic for $\lambda$, and  hence for $\om$ submanifolds.
\end{itemize} F. Laudenbach proved, see \cite{laudenbach},  that if  the   flow of $Z$ 
is Morse-Smale (i.e.  stable and unstable manifolds of critical points intersect transversely) and near critical points the vector field $Z$ is gradient with respect to an Euclidean metric, then the skeleton can be further Whitney substratified. It is likely that the Whitney condition also holds if near its  zeroes  the vector field  $Z$ is   gradient  with respect to any Riemannian metric. However, as far as know, this was never  verified in the literature.
The Whitney condition need not hold if eigenvalues of the linearization of $Z$ at critical points have non-vanishing imaginary parts, as a spiraling phenomenon of trajectories may occur.\footnote{I thank Francois Laudenbach for the discussion of the involved issues.}
 
Condition (L2)   holds for a  much more general class  of taming functions (e.g. when  $\phi$ is Morse-Bott), and hence for the 
  the purposes of this paper we will take the following working definition of a Weinstein structure, extending the class considered in \cite{CE12}:
{  {$\fW=( X,\lambda,Z, \phi)$ is {\em Weinstein}   if it satisfies conditions (L1) and (L2) with the Whitney condition
and also condition 
 \begin{itemize}
\item[(L3)] there exists a smooth family of Weinstein structures
 $\fW_t=( X,\lambda_t,\phi_t)$, $t\geq 0$ such that $(\lambda,\phi)=(\lambda_0,\phi_0)$ and  $\phi_t$ is Morse for $t>0$.
\end{itemize} } 

 \begin{problem} Which conditions (or maybe none?) on $\phi$  and $Z$ are needed to deduce (L2)  and (L3) from (L1)?
\end{problem} 

A. Oancea suggested to me  that a good sufficiently general condition on a Weinstein structure could be to require that  near critical points it is generated by a $J$-convex function with respect to some (not necessarily integrable) almost complex structure $J$, see \cite{CE12}, Chapter 1, for the details.

 \begin{remark}\label{rem:non-skeleton} Note that not  every closed
   subset  $C$ of a symplectic manifold which is   stratified by isotropic strata may serve  as the  skeleton for an appropriately chosen Weinstein structure on a neighborhood of $C$ (compatible with the given  ambient symplectic form). 
Examples of this kind exist already  in $\R^2$. 
 For instance, let $$C:=\{x=0,y\geq 0\}\cup\{x=y^2,y\geq0\}\cup\{y=0,x\geq 0\}\cup\{y=x^2,x\geq 0\}$$  be the union of 4 arcs emanating from the origin. Then there is no  Liouville structure on a neighborhood $U\ni 0$ which has   $C\cap U$  as a part of its skeleton.
  Indeed it is straightforward to check that  for any $1$-form $\lambda$ vanishing on $C\cap U$ one has $(d\lambda)_0=0$. 
  \end{remark} 
 
  \begin{problem}\label{prob:strat-Weinstein}
Find  a necessary and sufficient condition on a compact subset $C$ of a symplectic manifold to serve as the skeleton of some 
\begin{description}\item{a)} Liouville, or
\item{b)}  Weinstein
\end{description}
  structure on its neighborhood. In particular, is it true that a Whitney stratified subset $C$ which is the skeleton of a Liouville structure on its neighborhood also serves as the skeleton of a Weinstein structure?
 \end{problem}

\begin{figure}[h]
\begin{center}
\includegraphics[scale=.9]{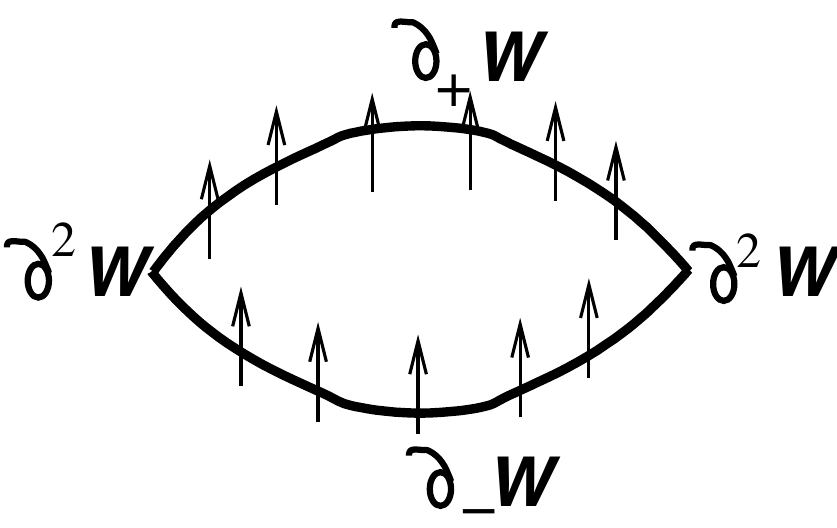}
\caption{Sutured Weinstein cobordism $W$ with corners.}\label{fig:corner}
\end{center}  
\end{figure}

 It is also
 useful to consider a  notion of  a {\em Weinstein cobordism}. This is a cobordism $(W,\p_-W=Y_-,\p_+W=Y_+)$ endowed with a Liouville form $\lambda$, whose  Liouville vector   field  $Z$  is outward transverse to $\p_+W$ and inward transverse to $\p_-W$, and a Lyapunov  (i.e. satisfying condition (L1)) function  $\phi:W\to\R$ for the   field $Z$. We  also postulate   (L3) and an analog of condition (L2) for  the core of  the  Weinstein cobordism, which we define in that case as     the stable manifold of the critical locus of $\phi$.  We will also be   considering 
 Weinstein cobordisms between manifolds with boundary $\p_\pm W$. We   will    view these cobordisms  as    sutured manifolds with  a  corner along the suture, see Fig. \ref{fig:corner}.
 More precisely, we assume that  the boundary  $\p W$ is presented as the union of two manifolds $\p_- W$ and $\p_+W$ with common boundary  $\p^2W=\p_+W\cap \p_-W$  along which it has a corner.   The vector field $Z$  transversely enters $W$  through $\p_-W$ and exits through $\p_+W$, but of course, in this case the function $\phi$ cannot be chosen constant on $\p_-W$ and $\p_+W$.

While any two Weinstein structures on the same symplectic  manifold are (canonically) homotopic as Liouville structures, the problem of existence of a {\em Weinstein} homotopy is widely  open.
\begin{problem} \label{prob:WH1}Let $(\wh X,\lambda_0,\phi_0)$ and $(\wh X,\lambda_1,\phi_1)$ be two completed Weinstein structures on the same symplectic manifold $(\wh X,\om)$. Are they homotopic as Weinstein structures? 
\end{problem}
In particular,
\begin{problem}
\label{prob:WH2}
Let $\fW=(\wh X,\om,\lambda,\phi)$ be a  completed  Weinstein structure, and $f:\wh X\to \wh X$ a symplectomorphism. Is the pull-back Weinstein structure
$f^*\fW$ is Weinstein homotopic to $\fW$?
\end{problem}
 The  Weinstein structure notion  was introduced in \cite{EliGro91}  as a symplectic counterpart of the notion of  Stein complex structure, and inspired   by the work of A.~Weinstein  \cite{Wei91},  see  also \cite{Eli90, Eli95,CE12}. 
 
I discussed the notions and problems considered in this paper with many people. I am especially grateful to
Daniel Alvarez-Gavela, Oleg Lazarev,  David Nadler, Sheel Ganatra,
Vivek Shende, Laura Starkston and Kyler Siegel for contributing many ideas and  suggestions for improvement of the current text. I am very grateful to the anonymous referee  for   critical remarks and many useful suggestions. Special thanks to Nikolai Mishachev for making the pictures.

\section{Weinstein hypersurfaces  and Weinstein pairs}
 Weinstein hypersurfaces are special cases of Liouville hypersurfaces introduced by Avdek in \cite{Avdek}.
 This and other related notions discussed in this paper are also  similar  to ``stops" of Sylvan, \cite{Sylvan} and  Liouville sectors of Ganatra-Pardon-Shende, \cite{GPS}. Related constructions are also considered in    Ekholm-Lekili's paper \cite{EkLe}.
 
 \subsection*{Weinstein hypersurfaces in a contact manifold}
Let $(Y,\xi)$ be a contact  manifold. A codimension 1 submanifold $\Sigma\subset  Y$ with boundary is called {\em Weinstein hypersurface} if there  exists a contact form $\lambda$ for $\xi$ such that $(\Sigma,\lambda|_{\Sigma})$ is compatible with a Weinstein structure on $\Sigma$, i.e.
$d\lambda|_\Sigma$ is symplectic and the Liouville vector field
$Z_\Sigma$ on $\Sigma$ dual to the Liouville form $\lambda|_{\Sigma}$  is outward transverse to $\p\Sigma$ and admits a Lyapunov function
 $\phi:\Sigma\to\R$.  The Reeb vector field for $\lambda$ is  transverse to $\Sigma$ and  the boundary $\p\Sigma$ of a Weinstein hypersurface $\Sigma$  is a codimension  two contact submanifold of $(Y,\xi)$.
 
 Though the induced Weinstein structure on $\Sigma$ depends on the choice of a contact form,    its {\em skeleton is independent of this choice}.
 Indeed, the Liouville fields for the Liouville  structures   $\lambda$ and $f\lambda$ for a positive $f>0$ are proportional. In fact, as it is computed in Lemma 12.1 in \cite{CE12} the form $f\lambda$ is Liouville if and only if $k:=\inf(f+df(Z))>0$, where $Z$ is the Liouville form for $\lambda$, and in that case the Liouville vector field for $f\lambda$ is equal to $\frac1kZ$.  Moreover, the space of  functions  $f$ for which $f\lambda$ is  Liouville (and hence in the considered case Weinstein) is contractible. 
   
 It follows that the  skeleton $\Core(\Sigma,\lambda|_{\Sigma})$   is a stratified subset of $Y$ which consists of  strata which are isotropic, and in the top dimension $n-1$ are Legendrian  for  the contact structure $\xi$.

  \begin{Example}
\begin{enumerate} \item
  {\em Weinstein thickening of a Legendrian submanifold.} Let $\Lambda\subset(Y,\xi)$ be a Legendrian submanifold. Then it admits a Darboux neighborhood $U(\Lambda)$ isomorphic to  $ (J^1(\Lambda),dz-pdq)$, $q\in\Lambda, ||p||^2+ z^2 \leq\eps^2$.
  Then $\Sigma(\Lambda):=U(\Lambda)\cap\{z=0\}$ is a Weinstein hypersurface symplectomorphic
  to the cotangent  ball bundle of $\Lambda$. Up to   Weinstein isotopy the Weinstein thickening $\Sigma(\Lambda)$ is independent of all the choices.  \footnote{Warning: unlike the case of a Legendrian isotopy, an isotopy of Weinstein hypersurfaces does not extend  in general to an ambient contact diffeotopy.}
   \item
  {\em Pages of open books.} According to Giroux's theorem \cite{Giroux-book}, any contact manifold admits an open book decomposition whose  pages  are Weinstein hypersurfaces.
  \item
   {\em Halves of convex hypersurfaces}. Recall that a hypersurface $\Sigma$ in a contact manifold is called {\em convex} if it admits a transverse contact vector field, see \cite{EliGro91, Giroux-convex}. The set $D$ of points where the contact vector field
  is tangent to the contact plane field, called a {\em dividing set}, is  generically  a smooth hypersurface which divides $\Sigma$ into two Liouville manifolds. In many interesting examples these Liouville  manifolds are, in fact, Weinstein, and hence serve a rich source of Weinstein hypersurfaces.
\end{enumerate}
  \end{Example}
  
  Given two Legendrian isotopic submanifolds $ \Lambda_0,\Lambda_1 \subset (Y,\xi)$ their Weinstein thickenings $\Sigma(\Lambda_0)$ and $\Sigma(\Lambda_1)$ are isotopic as Weinstein hypersurfaces. 
  
  \begin{problem}\label{prob:LvsW} Is the converse true?
  \end{problem}
  Here by isotopy we mean an isotopy of unparameterized submanifolds.
  
 Note that an isotopy of Weinstein hypersurfaces carries $\Lambda_0$ to an exact Lagrangian submanifold $\wt\Lambda_1\subset\Sigma(\Lambda_1)$. Moreover, there is a symplectomorphism
 $\psi:\Sigma(\Lambda_1)\to \Sigma(\Lambda_1)$ such that $\psi(\wt\Lambda_1)=\Lambda_1$.   Hence, the positive answer to Problem \ref{prob:LvsW}  would   follow
  from the positive resolution of the following special case of the  {\em nearby Lagrangian conjecture}: Lagrangians which are images of the $0$-section under a global symplectomorphism are Hamiltonian isotopic to  the $0$-section.
  
  If the contact manifold $(Y,\xi)$ is symplectically fillable then one can prove that
 the Legendrian algebras $LHA(\Lambda_0)$ and $LHA(\Lambda_1)$ are isomorphic\footnote{I thank Sheel Ganatra and Tobias Ekholm  for the discussion of this problem.}.  It is likely that this claim  could be generalized  to the case of a general contact manifold $(Y,\xi)$.

  \begin{problem} Is there an analog of  the Legendrian algebra $LHA(\Lambda)$ for a general Weinstein hypersurface? 
  \end{problem}

  Let us  return to the case of the Legendrian homology algebra of a Legendrian submanifold $\Lambda$ and  pick a contact form $\lambda$ such that its Reeb vector field   is tangent to the contact submanifold $\Delta:=\p\Sigma(\Lambda)$. We also choose an almost complex structure $J$ on $\xi$ such that $\xi\cap T(\Delta)$  are $J$-invariant. This allows us to define a deformation $(A[t], D)$ of the  Legendrian differential      algebra $ (A,\p):=LHA(\Lambda)$ as follows.
  For a generating chord $c\in A$ define   $D(c)=\sum\limits_{k\geq 0}(\p_kc)t^k$, where $\p_0=\p$ and    $\p_kc$   counts holomorphic curves with the intersection index $k$ with the symplectization of $\Delta$. This symplectization   is a complex hypersurface in the  symplectization of $Y$, and hence $k\geq 0$.  The sum defining differential $D$ is finite due to the  Gromov  compactness.
 \begin{problem}
  Explore whether the above  construction yields a genuinely new invariant of a Legendrian submanifold.
  \end{problem}

 Given a  Weinstein hypersurface $\Sigma\subset Y$ we slightly extend it to a larger Weinstein hypersurface $\wt\Sigma\supset \Sigma$ such that on $\wt\Sigma\setminus\Sigma$ the Liouville form $\lambda$ can be written as $t\lambda|_{\p\Sigma}$, $t\in[1,1+\eps]$. The  extended hypersurface $\wt\Sigma$ has  a neighborhood $\wt U$ diffeomorphic to $\wt\Sigma\times (-\eps,\eps)$ such that 
 $\lambda|_{\wt U}=\pi^*(\lambda|_{\wt\Sigma})+du$ where $u$ is the coordinate corresponding to the second factor and $\pi$ the projection $\wt U\to\wt \Sigma$.   Note that the level sets $\{u=\const\}$ are translates of $\wt \Sigma$ under the   Reeb flow of the contact form $\lambda$. Pick a non-negative  function $h:\wt\Sigma\to\R$ which is equal to $0$ on $\Sigma$ and to $t-1$ near $\p\wt\Sigma$ and set
 $U(\Sigma)=U_\eps(\Sigma):=\{h^2+u^2\leq\eps^2\}\subset\wt U$.
 The neighborhood $U(\Sigma)$  will be called the {\em  contact surrounding} of a Weinstein hypersurface $\Sigma$.

\begin{proposition}
\label{prop:open-cont} Contact manifolds
 $Y\setminus \overline{U(\Sigma)}$,  $Y\setminus \Sigma$  and $Y\setminus\Core(\Sigma,\lambda|_\Sigma)$ are contactomorphic.
\end{proposition}  
Let us first recall a few basic facts about convex hypersurfaces in contact manifolds.
If a germ  $\xi$ of a  contact structure along a closed hypersurface  $V$ in a $(2n-1)$-dimensional manifold   admits a transverse contact vector field $v$ then we canonically can construct a contact structure $\wh\xi$ on $V\times\R$ which is invariant with respect to translations along the second factor and whose germ along any slice $V\times t$, $t\in\R$,  is isomorphic to  $\xi$. 
We will  call  $\wh\xi $ the {\em invariant extension} of the convex germ $\xi$.

\begin{lemma}\label{lm:open-cont}
Let $V$ be a closed  $(2n-2)$-dimensional manifold and  $\xi$ a contact structure on  $ Y=V\times[0,\infty)$   which admits a contact vector field $v$  inward transverse to $V\times 0$ and such that its trajectories intersecting $V\times 0$  fill the whole manifold $Y$ (we do not require $v$ to be complete). Then $(Y,\xi)$ is contactomorphic  to $(V\times [0,\infty),\wh\xi)$, where $\wh\xi$ is the invariant extension of the germ of $\xi$ along $V\times0$. Moreover, for any compact set $C\subset Y$, $\Int C\supset V\times 0$, there exists a  contactomorphism $h:(Y,\xi)\to (V\times[0,\infty),\wh\xi)$ which is equal to the identity on $V\times 0$ and which sends the contact vector field $v|_C$ to  the vector field $\frac{\p}{\p t}$.
\end{lemma}

\begin{proof} It is sufficient to construct a {\em complete} contact vector field $\wt v$ on $V\times[0,1)$ which coincides with $v$ on $C$ and whose trajectories intersecting $V\times 0$ fill the whole manifold  $V\times[0,1)$. We will construct it using the following inductive process. Take a sequence of compact sets $C_0=C,  C_1,\dots,$ which exhausts $Y$, i.e.   $\bigcup\limits_0^\infty C_j=Y$ and  $C_j\subset \Int C_j$, $j=0,1,\dots$.
 Let  $v_1$ be a contact vector field obtained by cutting off $v$ outside $C_0$ but inside $C_1$. Let $h_1$ be the time $T_1>1$ flow map of $v_1$, where $T_1$ is chosen sufficiently large to ensure that $h_1(V\times 0)\subset C_1\setminus C_0$. Denote by  $\wt C_0$ the domain bounded by $V\times 0$ and $h_1(V\times 0)$ and by $\wt v_1$ the contact vector field equal to $v_1$ on $\wt C_0$ and to the  push-forward vector field $(h_1)_*v$ on $Y\setminus \wt C_0$.  Let $v_2$   be a contact vector field obtained by cutting off $\wt v_1$ outside $C_1$ but inside $C_2$ and denote by $h_2$ the time   $T_2>T_1+1$  flow of $v_2$, where  $T_2$ is chosen such that $h_2(V\times 0)\subset C_2\setminus C_1$.  Denote by  $\wt C_1$ the domain bounded by $V\times 0$ and $h_2(V\times 0)$ and by $\wt v_2$ the contact vector field equal to $v_2$ on $\wt C_1$ and to the  push-forward vector field $(h_1)_*\wt v_1$ on $Y\setminus \wt C_1$. Continuing this process we construct a sequence of contact vector fields $\wt v_1,\wt v_2,\dots,$ which stabilize on compact sets $C_1,C_2, \dots$ and converge to the contact vector field $\wt v$ on $Y$ with the required properties.
 \end{proof}
 \begin{proof}[Proof of Proposition \ref{prop:open-cont}]
 The contact vector field  $v =
 -Z_\Sigma-u\frac{\p}{\p u}$ is transverse to $\p U(\sigma)$ and  retracts $U(\Sigma)$ to $\Core(\Sigma,\lambda_\Sigma)$, and hence the contact structure on $U(\Sigma)\setminus \Core(\Sigma,\lambda_\Sigma)$ is canonically isomorphic to $\p U(\Sigma)\times[0,\infty)$ endowed with the  invariant extension $\wh\xi$ of the germ of contact structure $\xi$ along $\p U(\Sigma)$.  On the other hand,  $v$ is transverse to  $\p U_\delta(\sigma)$ for each $\delta\leq\eps $ and $\bigcup\limits_{\delta\in(0,\eps]}\p U_\delta(\Sigma)=U(\Sigma)\setminus\Sigma$. Hence, applying Lemma \ref{lm:open-cont} we conclude that $(U(\Sigma)\setminus\Sigma,\xi)$ is contactomorphic to $(\p U(\Sigma)\times[0,\infty),\wh\xi)$, and the claim follows. 
   
\end{proof}
  \begin{remark}  One of the corollaries  of Lemma \ref{lm:open-cont} is that
   {\em any open  domain in the standard contact $(\R^{2n+1}, dz+\sum x_idyi-y_idx_i)$ which is star-shaped with respect to the contact vector field $2z\frac{\p}{\p z}+\sum x_i\frac{\p}{\p x_i}+y_i\frac{\p}{\p y_i}$ is contactomorphic to $\R^{2n+1}$.} 
   On the other hand, in the standard contact $\R^3$ any open  domain diffeomorphic to $\R^3$ is contactomorphic to $\R^3$, see \cite{eliash-R3}.
    \end{remark}
  \begin{problem} Is there   a domain  in   the standard contact $\R^{2n+1}$, $n>2$,  which  is diffeomorphic to the closed ball, has      convex in contact sense  boundary, but whose interior is not contactomorphic to the standard $\R^{2n+1}$? Or even are there any open domains in    the standard contact $\R^{2n+1}$, $n>2$,
  which are diffeomorphic but not contactomorphic to $\R^{2n+1}$?
  \end{problem}

 \subsection*{Weinstein pairs}
 
 A  {\em Weinstein pair $(\fW,\Sigma)$} consists of a Weinstein domain $\fW=(  X,\lambda,\phi)$ together with a Weinstein  hypersurface
 $(\Sigma,\lambda|_\Sigma)$ in its  boundary $\p  X$. Equivalently, a Weinstein pair can be viewed as  a Weinstein manifold with cylindrical end, together with   a Weinstein hypersurface in its ideal contact boundary.
 
  Let $\Lambda=\Core(\Sigma)$ be the skeleton  of $\Sigma$ and
 $$\wh\Lambda:=\mathop{\bigcup}\limits_{t\geq 0}Z^{-t}(\Lambda)$$ be its saturation by the trajectories of the Liouville vector field $Z$. 
 The union
 $$\Core(  X,\Sigma):=\Core(X)\cup\wh\Lambda$$
 is called the {\em core}, or the {\em skeleton} of the Weinstein pair.
 
  It turns out that it is possible to   modify the Liouville form $\lambda$ on $ X$ in a neighborhood of $\Sigma$ in $X$  to make the attractor of the modified  Liouville vector field   equal to   the skeleton
$ \Core( X,\Sigma)$.

Given a Weinstein   pair $(\fW,\Sigma)$, $\fW=(X, \omega,\lambda,Z,\phi)$, let  $U=U(\Sigma)\subset\p X$
be   its  contact surrounding. Denote by $Z_\Sigma$ the Liouville field dual to $\lambda|_\Sigma$ and by $\phi_\Sigma$ its Lyapunov function. 
 A Liouville  form $\lambda_0$, the corresponding Liouville vector field   $Z_0$   for $\omega$ on $X$ and a  smooth function $\phi_0:X\to\R$ are called {\em adjusted } to the structure of the pair    if (see Fig. \ref{fig:adjusted})
 \begin{itemize}
 \item $Z_0$ is tangent to $\p X$ on $U(\Sigma)$ and transverse to $\p X$ elsewhere;
 \item $Z_0|_{U(\Sigma)}=Z_\Sigma+u\frac{\p}{\p u}$;
   \item  the attractor   $\mathop{\bigcap}\limits_{t\geq 0}Z_0^{-t}( X)$ of the Liouville  vector  field $-Z_0$ coincides with the core
$\Core( X,\Sigma)$ of the Weinstein pair;
 \item the  function $\phi_0:X\to\R$    is Lyapunov for $Z_0$ and   such that
 $\phi_0|_{U(\Sigma)}=\phi_\Sigma+u^2$ and $\phi_0$ has no critical values  $ \geq\eps^2=\phi_0|_{\p U(\Sigma)} $.
 \end{itemize}
 
  \begin{prop}\label{prop:pair-adjust} Given a Weinstein   pair $(\fW,\Sigma)$, $\fW=(X,\lambda,\phi)$, there exist a   Liouville form $\lambda_0$ for $\om$ and a function $\phi_0:X\to\R$ such that
  \begin{itemize}
  \item $\lambda_0,\phi_0$ are adjusted to $(\fW,\Sigma)$;
  \item $\lambda_0$   coincides with $\lambda$ outside  a neighborhood of $\Sigma$;
\end{itemize}  Moreover, there exists an extension $\wt \lambda,\wt\phi$ of  $(\lambda_0,\phi_0)$    to  a slightly bigger domain $\wt X\supset X$ such that the  $\wt\fW:=(\wt X,\wt\lambda,\wt \phi)$ is a Weinstein domain and  $\Core(\wt\fW)=\Core( \fW,\Sigma)$.
  \end{prop}
  
  To construct the adjusted Liouville field $Z_0$ let us write the form $\lambda$ near $\p X$ as $s(du+\lambda_\Sigma)$   near $U(\Sigma)$. Note that the Hamiltonian vector field $Y$ for a function $su$ near $U(\Sigma)$ coincides with $-s\frac{\p}{\p s}+u\frac{\p}{\p u}+Z_\Sigma$, and hence by appropriately cutting off the function
  $su$ outside    a neighborhood  of  $U(\Sigma)$ and subtracting the differential $dg$ of the resulting function $g$ to the Liouville form $\lambda$ we get the Liouville form $\lambda_0$ with the required properties.
   Note that the form $\lambda_0|_U$ is no more contact. Instead, $\lambda_0|_U=\pi^*(\lambda|_\Sigma)$.

Suppose that $ \lambda_0,\phi_0 $ are adjusted to the Weinstein pair $(\fW,\Sigma)$. Recall that $\phi_0|_{\p U(\Sigma)}=\eps^2$. Denote
$X_0=\{\phi_0\leq\eps^2\}$. We note that $\phi_0$ has no critical points in $X\setminus \Int X_0$, and hence $X_0$ is a manifold with boundary with a corner along $\p U(\Sigma)$ which is homeomorphic to $X$. We will sometime refer to $(X_0,\lambda_0,\phi_0)$ as the {\em cornered version} of the Weinstein pair
$(\fW,\Sigma)$.\footnote{The completion of the cornered version of a Weinstein pair is a special case of a Liouville sector in the sense of \cite{GPS}.}    For instance, the cornered version of the standard Weinstein ball $B^{2n}$ is the cotangent  ball bundle of $D^n$.
Thus, it is always possible to go back and forth between the original and adjusted (cornered) versions of a Weinstein pair, and we will  be using  the term ``Weinstein pair" for   both versions.

\begin{figure}[h]
\begin{center}
\includegraphics[scale=.6]{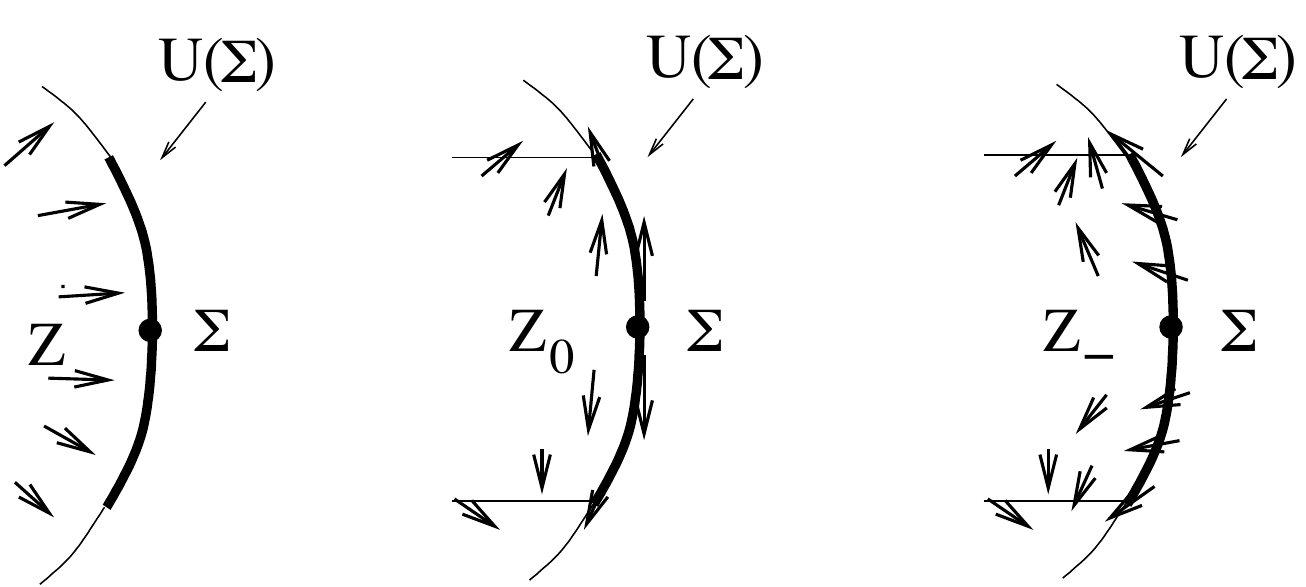}
\caption{Modifications of a Weinstein pair structure.}\label{fig:adjusted}
\end{center}  
\end{figure}

  \begin{remark}\label{rem:neg-pos}
   There are several other useful adjustments  of  a Weinstein pair structure. 
    Ekholm and Lekili in \cite{EkLe}, Section B.3,  are doing a similar to  the cornered version construction   by
 deforming the boundary $\p X$ near $U(\Sigma)$  without changing $Z$, as on Fig. \ref{fig:ELS}.  Without defining here Sylvan's stop structure we just say that for a given Weinstein pair there is a contractible space of choices of stop structures on the completion.
 
 \smallskip One can    also transform a Weinstein pair into a Weinstein cobordism whose negative boundary is $U(\Sigma)$, see Fig. \ref{fig:adjusted}:

{\em  Let $(X_0,\lambda_0,Z_0, \phi_0)$ be the cornered  adjusted  version of a Liouville pair structure $(\fW,\Sigma)$, as in  Proposition \ref{prop:pair-adjust}. There exists a  Liouville form $\lambda_-$ on $ X_0$  such that
   $(X_0, \lambda_-, \phi_0)$ is a sutured Weinstein cobordism structure with $\p_-X_0= U(\Sigma)$, and 
  $\Core(X_0,\lambda_-,\phi_0)=\Core(\fW,\Sigma)$.
 }
 
 To obtain such a form $\lambda_-$ one subtracts  from $\lambda$ the differential of the appropriately cut off function $2su$ instead of the function $su$ used to modify $\lambda$ into $\lambda_0$.
 
\begin{figure}[h]
\begin{center}
\includegraphics[scale=.6]{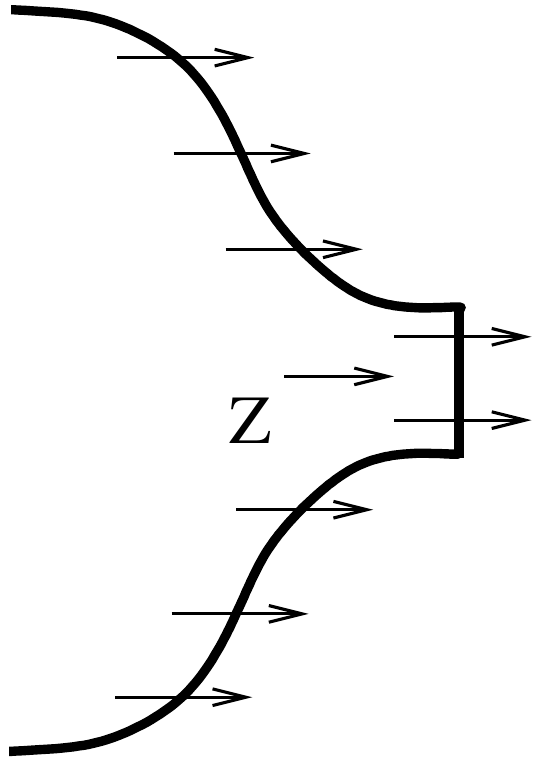}
\caption{Ekholm-Lekili deformation of $\p X$.}\label{fig:ELS}
\end{center}  
\end{figure}

\end{remark}

\section{Operations on Weinstein pairs} 
\subsection{Splitting and gluing of Weinstein pairs}\label{sec:gluing-splitting}
Let $\fW=(X,\lambda,Z,\phi)$ be a Weinstein domain.
A  hypersurface 
$(P,\p P)\subset (X,\p X)$ is called   {\em splitting} for  $\fW$ if it  satisfies the following conditions:
\begin{description}
\item{-} $\p P$ splits the boundary  $\p X$ into two parts, 
$\p X=Y_-\cup Y_+$ with $\p Y_-=\p Y_+=Y_+\cap Y_-=\p P$ (and respectively,  $P$ divides $X$ into two parts $X_+$ and $X_-$ with  
$\p X_-=P\cup Y_-, \p X_+=P\cup Y_+$ and $X_+\cap X_-=P$;
\item{-} the Liouville vector field $Z$ is tangent to $P$;
\item{-} there exists a    hypersurface $(S,\p S)\subset (P,\p P)$  which is Weinstein for the restricted Liouville form $\lambda|_S$, tangent to the vector field $Z$ and   intersects all leaves of the characteristic foliation  $\FF$ of the hypersurface $P$; we will refer to $S$ as the  {\em Weinstein soul} of the  splitting hypersurface $P$ and   denote it by $\So(P)$.
\end{description}
Note that the latter condition  together with  Lemma \ref{lm:open-cont} imply that $P$ is contactomorphic to the contact surrounding of its Weinstein soul.

 It follows that $(S,\lambda|_S,\phi|_S; \p S)$ is a codimension two  Weinstein  subdomain of $X$ and $\Core(S,\lambda|_S,\phi|_S)=\Core(\fW)\cap P$.
Moreover,  $(\fW_\pm;S)$, where  $\fW_\pm:=(X_\pm,\lambda|_{X_\pm},\phi|_{X_\pm})$ are cornered   Weinstein pairs and $\Core(\fW_\pm;S)=\Core(\fW)\cap X_\pm.$

The {\em gluing} construction  reverses the splitting. This operation was  considered by Avdek in \cite{Avdek} in the context of Liouville hypersurfaces.
 Let $(\fW,\Sigma)$ and  $(\fW',\Sigma')$ be two Weinstein pairs and $(X_0,\lambda_0,\phi_0)$, $(X_0',\lambda'_0,\phi'_0)$    their cornered forms.   Let  $F:(\Sigma,\lambda|_\Sigma,\phi|_{\Sigma})\to
(\Sigma',\lambda'|_{\Sigma'},\phi'|_{\Sigma'})$ be a Weinstein isomorphism. We extend $F$ to   a contactomorphism $ U(\Sigma)\to U(\Sigma')$, still denoted by $F$, and use it to define  a  domain   $$X\mathop{\sqcup}\limits_F X':=X_0\sqcup X'_0/\{(x\in U(\Sigma))\sim( F(x)\in U(\Sigma')).$$ Then the Liouville forms $\lambda_0$ and $\lambda_0'$, as well as  Lyapunov functions  $\phi_0:X_0\to\R$ and $\phi_0':X_0'\to\R$, can be  glued together to define a Weinstein structure $  (\fW,\Sigma)\mathop{\cup}\limits_F(\fW',\Sigma'):=(X_F,\lambda_F,\phi_F)$,
see Fig. \ref{fig:gluing}.

\begin{figure}[h]
\begin{center}
\includegraphics[scale=.8]{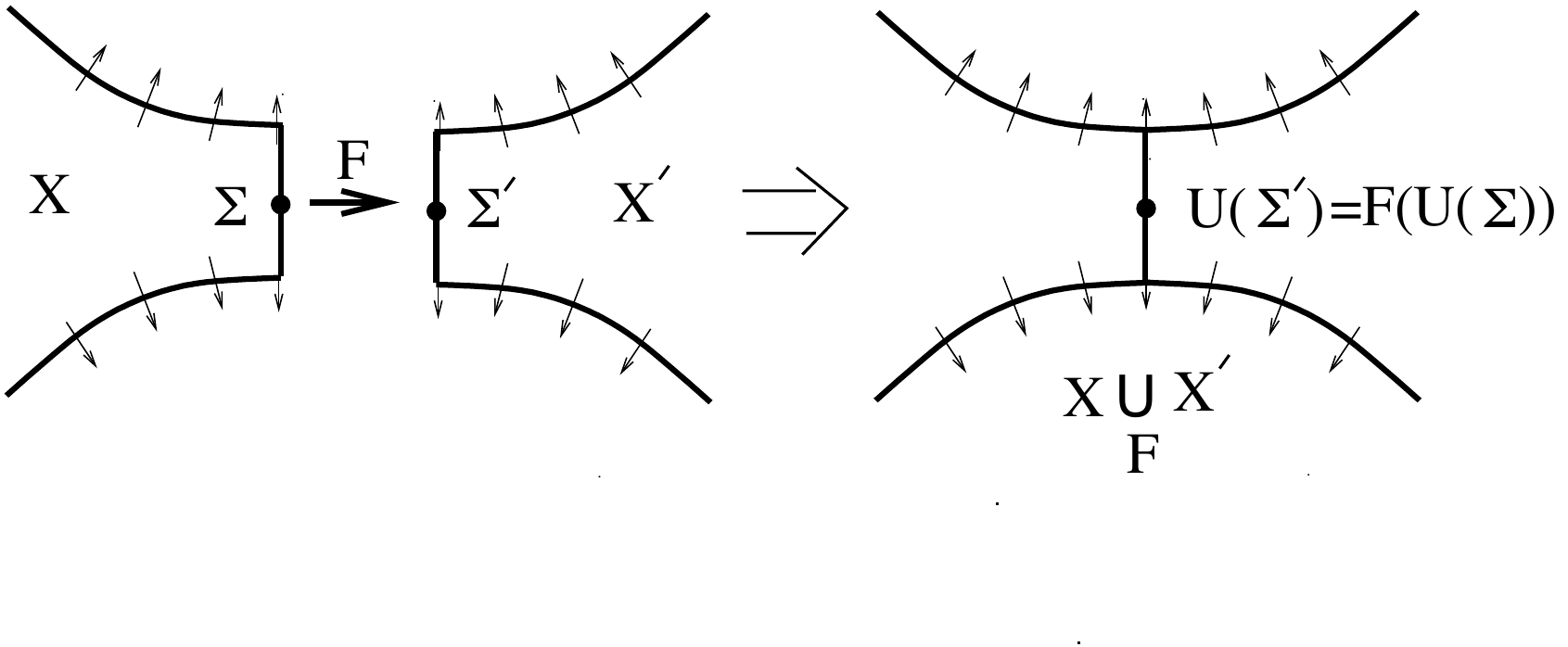}
\caption{Gluing of Weinstein pairs.}\label{fig:gluing}
\end{center}  
\end{figure}
 Note that $$\Core(X_F,\lambda_F,\phi_F)=\Core(X,\Sigma)
 \mathop{\cup}\limits_{F|_{\Core(\Sigma)}}\Core(X',\Sigma').$$

Note that  the constructed  Weinstein domain $X_F$ contains $U(\Sigma)$ as its splitting hypersurface. Applying the above  described splitting construction we get back the   Weinstein pairs
 $(\fW,\Sigma)$ and  $(\fW',\Sigma')$.

The gluing of Weinstein pairs   is a generalization of  the Legendrian surgery  construction (or rather Weinstein handle attachment). When $\Sigma=\Sigma(\Lambda)$ for a Legendrian $\Lambda\subset \p X$, $X'=B^{2n}$ and $\Sigma'=\Sigma(\Lambda_0)$, where $\Lambda_0$ is the Legendrian unknot in $S^{2n-1}=\p B$, then $(X_F,\lambda_F,\phi_F)$ is the Weinstein $n$-handle attachment to $X$  along $\Lambda$. Conversely,  the general   gluing operation $  (\fW,\Sigma)\mathop{\cup}\limits_F(\fW',\Sigma')$  can be decomposed     into  a sequence of subcritical and critical handle attachments. To do  that, one fixes  first a Weinstein handle decomposition of $\Sigma$, and then  for each  handle of index $k$  of this decomposition one needs to attach   a handle of index $k+1$ to  the glued domains. For instance, for a handle of index $0$ centered at a point $p\in\Sigma$ one attaches a handle of index $1$    along an arc connecting the point $p\in \Sigma$ with its image $p'=F(p)\in\Sigma'$ under the gluing map. 

Both, splitting and gluing constructions can be naturally generalized to the relative setting.  Let $(\fW,\Sigma)$ and  $(\fW',\Sigma')$ be two Weinstein pairs. Suppose that $\Sigma$ and $\Sigma'$ are split by splitting hypersurfaces $T\subset \Sigma$ and $T'\subset \Sigma'$ as $\Sigma=\Sigma_-\cup\Sigma_+$ and 
 $\Sigma'=\Sigma'_-\cup\Sigma'_+$ and we are given a Weinstein isomorphism $F:\Sigma_+\to\Sigma_-'$. Then the result of 
  the  {\em partial  gluing} is the pair
  $  (\fW,\Sigma)\mathop{\cup}\limits_{F,T,T'}(\fW',\Sigma')$ which consists of the Weinstein domain   $(\fW,\Sigma_-)\mathop{\cup}\limits_F(\fW',\Sigma'_ +)$ together with the Weinstein hypersurface
  $\Sigma_-\mathop{\cup}\limits _{F|_{\So(T)}}\Sigma_+'\subset X\mathop{\cup}\limits_F X'$ which is   the result of gluing the Weinstein pairs $(\Sigma_+,\So(T))$ and $(\Sigma'_-,\So(T'))$ using the  Weinstein isomorphism $F|_{\So(T)}$.

The reverse operation to the partial gluing of Weinstein pairs is a splitting of a Weinstein pair  $(\fW,\Sigma)$, $\fW=(X,\lambda,\phi)$,  along a   splitting hypersurface
$(P,Q:=\p P)\subset (X,\p X)$    for the Weinstein domain $X$     where  in addition  $P$ satisfies the following condition:

\begin{description}
\item{-} $Q$ intersects $\Sigma$ transversely, $Q\cap\Sigma=\So(Q)$ and $Q\cap \Sigma$ is a splitting hypersurface for   $\Sigma$, which splits  it into  $\Sigma_+$ and $\Sigma_-$; 
  \end{description}
 
 The result of this splitting are   two Weinstein pairs 
 $(X_-,\wt\Sigma_-)$ and $(X_+,\wt\Sigma_+)$, where   the Weinstein hypersurface $\wt\Sigma_\pm\subset\p X_\pm=Y_\pm\cup P$ is the result of gluing of Weinstein pairs $(\Sigma_\pm,\So(Q\cap\Sigma))$ and $(\So(P),\So(Q\cap\Sigma))$.
     
     As in the absolute case,
     the gluing operation of Weinstein pairs glues their skeleta along the skeleta of glued hypersurfaces.
Conversely, a  splitting of the skeleton of a Weinstein domain 
lifts to a splitting of a Weinstein domain into two Weinstein pairs.
  
\subsection{Product and Stabilization of Weinstein pairs}\label{sec:prod-stab}
Given two Weinstein pairs $(\fW,\Sigma)$ and $(\fW',\Sigma')$, where
$\fW=(X,\lambda,\phi)$, $\fW'=(X',\lambda',\phi')$ we define their product
 as the Weinstein pair 
 $$(\fW,\Sigma)\times(\fW',\Sigma'):=(X\times X',\lambda\oplus\lambda', (\Sigma\times X';\Sigma\times\Sigma')\mathop{\sqcup}\limits_\Id (X\times\Sigma', \Sigma\times\Sigma')).$$

Here  $(\Sigma\times X';\Sigma\times\Sigma')\mathop{\sqcup}\limits_\Id(X\times\Sigma', \Sigma\times\Sigma'))$ is the result of gluing of two Weinstein pairs by the identity map between the  Weinstein hypersurfaces
$\Sigma\times\Sigma'\subset\p(X\times\Sigma')$ and $\Sigma\times\Sigma'\subset\p(\Sigma\times X').$
We note
that
$$ \Core\left((\fW,\Sigma)\times(\fW',\Sigma')\right)=\Core(\fW,\Sigma)\times\Core(\fW',\Sigma').$$
 
In the case when $(X',\Sigma')$ is the Weinstein pair $(T^*D^k,T^*S^{k-1})$ the product operation is called the {\em stabilization} (or $k$-stabilization). It was first proposed in  a slightly different form by M. Kontsevich, \cite{Konts-stab}. The core of the  $k$-stabilized pair $(\fW,\Sigma)$  is  equal
to $\Core(\fW,\Sigma)\times D^k$.
 
It is important to stress the point that the result of the stabilization  is always a Weinstein pair with a non-empty hypersurface in the boundary, even if we begin with the absolute case of a Weinstein domain.

\subsection{Weinstein homotopy as a Weinstein pair}\label{sec:hom-pair}

  Consider a Weinstein structure $\fW_0:=(X,\om, \lambda_0,\phi_0)$ and its $1$-stabilization 
  $\fW^\st:=\fW\times T^*I$, viewed as a Weinstein pair $(X\times T^*I,\lambda_0+udt, X\times 0\cup X\times 1)$.
 Consider  a  Weinstein homotopy $\fW_t:=(X,\ \lambda_t=\lambda_0+dh_t,\phi_t)$, $t\in[0,1]$. We assume, in addition, that $\dot h_1=\dot h_0=0, $ where we denoted  $\dot h_t:=\frac{d{h_t}}{dt}(t).$ This condition can always be arranged by a re-parameterization of the homotopy.
 Consider  the product  $X\times T^*I$ with the symplectic form $\Omega:=\om\oplus du\wedge dt$ , where $(u,t)$ are canonical coordinates on $T^*I$ (so that $u=0$ defines the $0$-section). Note that the $1$-form  $\wt\lambda:=\lambda_t+(u+\dot h_t)dt$ is  a Liouville form for $\Omega$. Indeed,   $d\wt\lambda:=d\lambda_t +dt\wedge d\dot h_t+d\dot h_t\wedge dt+du\wedge dt=\om+du\wedge dt.$ 
 
 We  have $\wt\lambda|_{X_0}=\lambda_0$ and $\wt\lambda|_{X_1}=\lambda_1$. 
 \begin{proposition}
There  exists a function $\wt\phi=X\times T^*I\to\R$ such that 
 \begin{align*}
 &\left((X\times T^*I,
 \wt\lambda:=\lambda_t+(u+\dot h_t)dt,\wt\phi; X_0\cup X_1\right)\\
 &\hbox{where}\;\; X_0:=X\times\{t=u=0\}, \;;  X_1:=X\times \{t=1, u=0\}  ,
 \end{align*}  
   is 
a  Weinstein  pair.
\end{proposition} 
We call this pair the {\em  concordance   generated by the homotopy $\fW_t$}.
 \begin{proof}  Note that the corresponding to $\wt\lambda$ Liouville vector field is given by the formula   $\wt Z= Z_t+(u+\dot h_t)\frac{\p}{\p u},$  where $Z_t$ is the Liouville vector field corresponding to $\lambda_t$.
 Define the function $\wt\phi$ by the formula
 $\wt\phi=\phi_t+\frac k2(u+\dot h_t)^2$, where a positive constant $k$ will be chosen later.  
 Then we
 have
 \begin{align*}
 d\wt\phi(\wt Z)=d\phi_t(Z_t)+k(u+\dot h_t)^2+k(u+\dot h_t)d\dot h_t(Z_t). 
 \end{align*}
 Not that $|d\phi_t(Z_t)\geq a||Z_t||^2$ and $|d\dot h_t(Z_t)|\leq b||Z_t||$ for some constants $a,b>0$. Denoting $X:=||Z_t||, Y:=u+\dot h_t$ we can write
  \begin{align*}
 &  |d\wt\phi(\wt Z)|\geq   a||Z_t||^2+k(u+\dot h_t)^2-bk|u+\dot h_t| ||Z_t||  \\
 &aX^2+kY^2-bkXY.
  \end{align*}
  The quadratic form  $aX^2+kY^2-bkXY$ is positive definite if $b^2k^2-4ak<0$ or $k<\frac{4a}{b^2}$. Under  this condition, which   can be arranged by choosing the constant $k$ sufficiently small,  
  we get
  $ |d\wt\phi(\wt Z)|\geq c(X^2+Y^2)\geq \wt c||\wt Z||^2$ for  positive constants  $c, \wt c$. This concludes the proof.
 \end{proof}
 \begin{remark}
 The critical point locus of $\wt \phi$ ($=$ the zero locus of $\wt Z$) is equal to
 $$\wt C=\{(x,t,u);\; x\;\;\hbox{is a critical point of}\;\; \phi_t,\;u=\dot h_t(x), t\in[0,1]\}.$$
The stable manifold  of a critical point $(x_0,t_0,u_0)$ projects to the stable manifold of the critical point $x_0$ of $\phi_{t_0}$. Its $u$-coordinate can be found  by solving the inhomogeneous linear ODE $$\frac {du(\gamma(s))}{ds}=u(\gamma(s))+\dot h_{t_0}(\gamma(s))$$ with the asymptotic boundary condition $\mathop{\lim}\limits_{s\to\infty}u(\gamma(s))=u_0$, where $\gamma(s)$ is a trajectory of $X_{t_0}$ converging to the critical point $x_0$.                     
 \end{remark}
  
    \section{Looseness and Flexibility}
Let us recall  that  in   contact manifolds of dimension $2n-1\geq 5$ there is  a local modification construction for  Legendrian submanifolds, called {\em stabilization}\footnote{The term ``stabilization" is used here in a completely different sense than in Section \ref{sec:prod-stab}.} , see \cite{Eli90,M12, CE12}. This operation can be performed  in an arbitrarily small neighborhood of  any point of a Legendrian. Moreover,  it can also be performed without changing the formal Legendrian isotopy class of the Legendrian submanifold.
  In her 2012 paper \cite{M12} Emmy Murphy called a   Legendrian submanifold  {\em loose} if it is isotopic to a stabilization of another Legendrian submanifold, and showed    that loose Legendrians satisfy an $h$-principle: any two loose formally isotopic Legendrians can be connected by a Legendrian isotopy.
 
The notion  of {\em flexibility}, see \cite{CE12}, for Weinstein cobordisms is tightly related to the looseness property of Legendrian knots. One  first defines flexibility for  {\em elementary}
 Weinstein cobordisms, i.e. Weinstein cobordisms    $(W,\om,Z,\phi)$ without any  $Z$-trajectories connecting critical points  of the Lyapunov function $\phi$.
An elementary $2n$-dimensional, 
$n>2$, Weinstein cobordism  
$(W,\om,X,\phi)$ is called 
{\em flexible} 
if the attaching spheres of all
index $n$ handles form in $\p_-W$ a {\it loose} Legendrian link (i.e. each sphere is loose in the complement of the others). 
A Weinstein structure is called {\em flexible} if it is homotopic to one  which can be decomposed into elementary flexible
cobordisms.

 As it was shown by E. Murphy and K. Siegel in \cite{MS15} existence of  a decomposition into flexible elementary cobordisms really depends on the choice of a particular Weinstein structure in the given homotopy class. Moreover,   there exist non-flexible Weinstein domains which  become flexible after attaching an $n$-handle.

Flexible Weinstein structures are indeed flexible: they abide a number of $h$-principles.
\begin{theorem}\label{thm:Weinstein-flexible}
\begin{enumerate}
\item {\rm(\cite{CE12})} Any two flexible Weinstein structures on a given smooth cobordism are homotopic as Weinstein structures provided that the corresponding symplectic forms are in the same homotopy class of non-degenerate (but not necessarily closed) $2$-forms.
\item {\rm(\cite{CE12})}  Let $(X,\om, Z,\phi)$ be any flexible Weinstein structure and $\phi_t$, $t\in[0,1]$, be a   family of generalized Morse functions such that $\phi_0=\phi$. Then there exists a homotopy 
 $(X,\om_t, Z_t,\phi_t)$ of Weinstein structures.
 \item {\rm(\cite{EM-caps})} Let $(X_\pm,\om_\pm, Z_\pm,\phi_\pm )$ be two  Weinstein structures. Suppose that the structure   
 $(X_-,\om_-, Z_-,\phi_- )$ is  flexible and that there exists an embedding
 $f:X_-\to X_+$ such that the forms $\om_-$ and $f^*\om_+$ are homotopic as non-degenerate (but not necessarily closed) $2$-forms. Then there exists a homotopy of Weinstein structures
 $(X_-,\om^t_-, Z^t_-,\phi^t_- )$, $t\in[0,1]$,  beginning  with 
 $(X_-,\om^0_-=\om_-, Z^0_-=Z_-,\phi^0_-=\phi_- )$ and an isotopy $f^t:X_-\to X_+$ beginning with $f^0=f$ such that $(f^1)^*\om_+=\om_-^1$.
\end{enumerate}
\end{theorem}
   At first glance Theorem  \ref{thm:Weinstein-flexible} implies that symplectic topology of flexible Weinstein manifolds is quite boring. This is also confirmed by the fact that symplectic homology  in all its flavors  of a flexible Weinstein manifold is trivial. However, as we will see below in Section \ref{sec:filing} the contact boundaries of flexible Weinstein domains have a rich contact topology.

 The looseness property of  a Legendrian submanifold can be naturally extended  to  Weinstein hypersurfaces of contact manifolds. A Weinstein hypersurface $\Sigma $ of a contact manifold $Y$ of dimension $2n+1\geq 5$ is called
 {\em loose} if   for each    $n$-dimensional strata  $S$ of the skeleton $\Core(\Sigma)$ there is a ball $B_S\subset Y\setminus (\Core(\Sigma)\setminus S)$   such that $B_S\cap S$ is loose in $B_S$ relative $\p(B_S\cap S)$.
   A canonical Weinstein thickening of a loose Legendrian knot is loose. However, it is unclear whether looseness is preserved under   Weinstein isotopy.
 
 \begin{problem}  Is looseness property preserved under a Weinstein isotopy of $\Sigma$. In particular,  suppose that a Weinstein thickening  $\Sigma(\Lambda)$ of a Legendrian knot $\Lambda$ is isotopic in the class of Weinstein hypersurfaces to a loose Weinstein hypersurface.  Does this imply that $\Lambda$ itself   is  loose?
 \end{problem}
 \begin{prop} 
 Let $(\fW,\Sigma)$ and $(\fW',\Sigma')$ where 
$\fW=(X,\lambda,\phi)$, $\fW'=(X',\lambda',\phi')$,  be two  Weinstein pairs and  $$(\fW,\Sigma)\times(\fW',\Sigma'):=(X\times X',\lambda\oplus\lambda', \wt\Sigma:=(\Sigma\times X';\Sigma\times\Sigma')\mathop{\sqcup}\limits_\Id (X\times\Sigma', \Sigma\times\Sigma'))$$  be their product. 
Suppose that  $\Sigma$ is loose in $\p X$. Then $ \wt\Sigma $ is  loose in  $\p (X\times X')$.
 \end{prop}
Indeed, this is straightforward from the following fact:  given any contact manifold $(Y,\{\alpha=0\})$,  a Liouville manifold $(U,\mu)$,  a loose Legendrian $\Lambda\subset  Y$ and
a Lagrangian $L\subset  U$ with $\mu|_L=0$, then the Legendrian $\Lambda\times
L\subset  (Y\times U,\{\alpha\oplus\mu=0\})$ is loose as well. 
 
Let us stress the point that while flexibility of a Weinstein manifold is its intrinsic property, the looseness  of a Weinstein hypersurface depends on its embedding in the contact manifold.
However,  the above fact about the looseness of a product shows  that flexibility always implies looseness (I thank the referee for this argiment).
 \begin{prop}
 Let $(Y,\xi)$ be a contact manifold of dimension $\geq 7$, and $\Sigma\subset Y$ a flexible  Weinstein hypersurface. Then $\Sigma$ is loose.
 \end{prop}
 Indeed,   let $\alpha$ be a contact form for $\xi$ which restricts to a Liouville form $\mu$ on $\Sigma$. Consider a Weinstein  subdomain $\Sigma_0\subset\Sigma$ and let  a Lagrangian disc $\Delta\subset \Sigma\setminus\Sigma_0$ be attached to $\Sigma_0$ along a loose Legendrian sphere $\Lambda:=\p\Delta\subset\p\Sigma_0$. In a neighborhood   $U\supset\p\Sigma_0$ in $\Sigma$ the Liouville form $\mu$   can be written as $s\beta$, $s\in(1-\eps, 1+\eps)$ for a contact form on $\p\Sigma_0$, and on a neighborhood $\wt U$ of $U$ in $Y$ the contact form $\alpha$ can be written as $dt+s\beta=s(udt+\beta)$, $|t|<\eps, u=\frac1s$. 
 Hence, $\wt U$ can be viewed as the product of the contact manifold  $(\p\Sigma_0,\beta)$ and a Liouville subdomain  $$ Q:=\{(u,t)\in\left(-\frac 1{1+\eps},\frac1{1-\eps}\right)\times(-\eps,\eps)\}\subset (\R^2,udt),$$ while $\Delta\cap \wt U=\Lambda\times\{t=0, \frac1{1-\eps}<u\leq 1] \}\subset \Sigma_0\times Q$. Hence looseness of  attaching spheres of  top index Weinstein handles of $\Sigma$ implies looseness of their Lagrangian cores viewed as Legendrian submanifolds of $Y$.
 \smallskip
 
 The notion of flexibility naturally extends to Weinstein pairs.
 A Weinstein pair $(\fW=(X,\lambda,\phi),\Sigma)$ is called {\em flexible}  if it is flexible viewed as a cobordism between $\p X_-=U(\Sigma)$ and $\p_+X=X\setminus\Int U(\Sigma)$, see Remark \ref{rem:neg-pos}.
 It is straightforward to see that flexibility is preserved under the stabilization construction. However, the converse is not clear.
 \begin{problem}\label{prob:destab-flex}
 Suppose that the stabilization of a Weinstein pair is flexible. Does this imply that the Weinstein pair itself is flexible? More generally, does existence of a homotopy between stabilizations of two Weinstein (pair) structures implies existence of a homotopy between the structures themselves?
 \end{problem}
 
 Attaching a critical handle along a loose Legendrian knot  to a flexible Weinstein domain by definition preserves its flexibility. This generalizes to the following
 \begin{prop}\label{prop:flex-gluing}
 Let $(\fW=(X,\lambda,\phi),\Sigma)$ and $(\fW'=(X',\lambda',\phi'),\Sigma')$ be  two  Weinstein pairs. Let 
 $\Sigma$ and $\Sigma'$  be decomposed as $\Sigma=\Sigma_-\cup\Sigma_+, \Sigma'=\Sigma'_-\cup\Sigma'_+$ by splitting hypersurfaces $T\subset\Sigma$ and $T'\subset\Sigma'$, see Section \ref{sec:gluing-splitting}. Suppose that 
 \begin{description}
 \item{-} there exists a Weinstein isomorphism $F:\Sigma_+\to\Sigma'_-$,  \item{-}
  $\Sigma_-$ is loose in $\p X$ and 
  \item{-} pairs $(\fW,\Sigma_-)$ and $(\fW',\Sigma_+')$ are flexible.
  \end{description}
   Then the glued pair   $  (\fW,\Sigma)\mathop{\cup}\limits_{F,T,T'}(\fW',\Sigma')$  is flexible. In particular, the result of gluing of two flexible Weinstein domains along Weinstein hypersurfaces  one of which is loose is flexible.
  \end{prop}
  This follows from the fact that the gluing operations of two Weinstein pairs can be decomposed into a sequence of handle attachments, and the looseness assumption  for the Weinstein hypersurface in one of the glued parts implies that all the critical handles are attached along loose knots.
  
  As a corollary Proposition \ref{prop:flex-gluing} implies the following
  generalization of the following result of   E. Murphy and K. Siegel, \cite{MS15}:
 \begin{prop}\label{prop:prod-flex}
 The product of two Weinstein pairs, one of which is flexible, is flexible.
 \end{prop}
 Indeed,  the product of  two Weinstein pairs can   always be  built by  a sequence of gluing of various stabilizations of the first pair.
   \section{Lagrangian submanifolds of Weinstein domains}
   In this section we discuss {\em exact} Lagrangian submanifolds in 
  a Weinstein domain $(X, \lambda,\phi)$. The Lagrangians will always be assumed either closed or with Legendrian  boundary in $\p L\subset \p X$. 
  
Let $\Sigma(\p L)$ be the Weinstein thickening of the (possibly empty) Legendrian boundary $\p L$.
A Lagrangian $L$ is called {\em regular}, see \cite{EGL-flex}, if the Weinstein pair $(X,\Sigma(\p L))$ admits a skeleton which contains $L$.

\begin{problem} Are there non-regular exact Lagrangians?
\end{problem}

The  problem is widely open. While no   examples of non-regular  Lagrangians are known, in the opposite direction  in the case of a closed exact Lagrangian $L$ in a general Weinstein domain $X$  it is   even unknown whether $L$ realizes a non-zero homology class in $H_n(X)$ (which is a necessary condition for its regularity).
 
 If $L\subset X$ is regular then by removing its tubular neighborhood $N(L)$ one gets a Weinstein cobordism  $X_L:=(W\setminus N(L),\p_-X_L:=\overline{\p N(L)\setminus\p X},\p_+X_L:=\overline{\p X\setminus N(L)})$ (between manifolds with boundary if $\p L\neq \varnothing$) whose negative boundary is the unit cotangent bundle of $L$.
 The Lagrangian $L$ is called {\em flexible},  see \cite{EGL-flex}), if the cobordism $X_L$ is flexible.
 
It was shown in \cite{EGL-flex} that  any flexible  $(X,\lambda)$  admits  a surprising abundance of flexible Lagrangians with non-empty Legendrian boundary. In particular,

\begin{theorem}\label{cor-3-manifold} 
Let $L$ be an  $n$-manifold with
non-empty boundary, equipped with a fixed trivialization $\eta$ of its
complexified tangent bundle $TL\otimes\C$. 
Then there exists a flexible Lagrangian embedding with Legendrian boundary
$(L,\p L)\to (B^{2n},\p B^{2n})$ where $B^{2n}$ is the standard symplectic
$2n$-ball,  realizing the trivialization $\eta$. In particular, any 
  $3$-manifold with boundary can be realized as a flexible Lagrangian submanifold of $B^6$ with Legendrian boundary in $\p B^6$. 
\end{theorem}

\section{Symplectic topology of Weinstein manifolds}
While flexible Weinstein structures enjoy a   full parametric $h$-principle,  there is plenty of symplectic rigidity and fine symplectic invariants of non-flexible ones. I will not discuss in this survey any such invariants and just mention that until recently most  examples  of formally homotopic but not symplectomorphic Weinstein manifolds were distinguished by their (possibly appropriately  deformed)  symplectic cohomology.
For instance, there are infinitely many non-symplectomorphic Weinstein structures on $\R^{2n}$
for any $n>2$  
(\cite{SS-example, mclean}) and by taking connected sums of  these examples with flexible Weinstein manifolds one gets infinitely many non-symplectomorphic Weinstein structures on any given ``almost Weinstein" (i.e. an almost complex manifold of homotopy type of a half-dimensional CW-complex) manifolds, see
 \cite{abouzaidseidel}.

Note that Theorem \ref{cor-3-manifold}  can   also  be  used for constructing   exotic Weinstein structures.
  In particular, 
\begin{theorem}[\cite{EGL-flex}]\label{thm:closed-exotic}
Let $L$ be a closed    $3$-manifold.
Then there exists a unique up so symplectomorphism Weinstein structure $\fW(L)= (\om_L,Z_L,\phi_L)$ on $T^*S^3$  which contains $L$ as its
flexible Lagrangian submanifold in the homology class of the $0$-section (with $\Z/2$-coefficients in the non-orientable case).
Moreover, infinitely many of these $\fW(L)$ are pairwise non-symplectomorphic.
\end{theorem}
Note that there exists only 1 homotopy class of almost complex structures on $T^*S^3$.

While the symplectic structure of $\fW(L)$ carries a lot of information about the topology of $L$, the following problem is open:
\begin{problem}\label{prob:LagS3}
Suppose $\fW(L)$ is symplectomorphic to $\fW(L')$? Does it imply that $L$ is diffeomorphic to $L'$?
\end{problem}
The famous "nearby Lagrangian problem" asks whether there is a unique
up to Hamiltonian isotopy exact closed Lagrangian submanifold in the standard $T^*M $ for a closed  $M$. Though in this form the answer is unknown except for $M=S^2$ and $T^2$,  see \cite{Hind, DRGI}, the answer is positive up to simple homotopy equivalence, \cite{AbKra}, and hence according to Smale, Freedman and Perelman  for $M=S^n$ up to homeomorphism,    and for some dimensions, e.g. $n=3,5,6,12,$ even up to diffeomorphism, \cite{Milnor-Kervaire}.  As it was pointed out to me  by O. Lazarev, one can show using methods of \cite{CaGaHaSa} that certain  exotic $T^*S^n$ may contain several not homotopy equivalent regular  closed exact  Lagrangian submanifolds.
\begin{problem}\label{prob:nearby}
Can the uniquenes results from \cite{AbKra} be extended to a more general class of Weinstein  structures on $T^*S^n$?
\end{problem}

  The proof of  Theorem \ref{cor-3-manifold} yields also the following slightly stronger result.
\begin{thm}\label{cor-3-manifold2}
Let $(X,\om, \lambda,\phi)$ be a $6$-dimensional  Weinstein   domain  such that $\phi$ has exactly  1 critical point of index $3$ (and any number of critical points of  smaller indices). Suppose also that  the   symplectic vector bundle $(TX, d\lambda)$ is trivial. Then there exists a Weinstein structure  $(\om_X,\lambda_X,\phi_X)$ on $T^*S^3$  which admits an embedding $$(X,\om, \lambda,\phi)\to  (T^*S^3,\om_X, \lambda_X,\phi_X)$$ onto a  Weinstein subdomain with a flexible complement.
\end{thm}

\section{Topology of Weinstein fillings}\label{sec:filing}
Contact manifolds appeared as boundaries of Weinstein domains  are called {\em Weinstein fillable}.
The fact that a Weinstein filling has a homotopy type of a half-dimensional CW-complex imposes constraints on the topology of its contact boundary and the stable almost complex class which can be realized by   Weinstein fillable contact structures on a given smooth manifold.
This question was studied in detail by Bowden-Crowley-Stipsicz in \cite{BCS1, BCS2}.  In particular, they showed that there are classes of homotopy spheres which do not admit any Weinstein fillable contact structure.

Given a contact manifold $(Y,\xi)$ one can try to describe (symplectic) topology of its Weinstein fillings. In this section we discuss this problem for contact manifolds of dimension $2n-1>3$, see \cite{Ozbagci} for a survey of results for $3$-dimensional manifolds.

First of all notice that the fact that   $X$ retracts to its $n$-dimensional skeleton implies that   the inclusion $Y=\p X\hookrightarrow X$ is $(n-1)$-connected, and in particular, if $Y$ is a homotopy sphere then $X$ is $(n-1)$-connected. It turns out that some  contact structures know  much more about the topology of their fillings.
\begin{theorem}[\cite{McD-FE}]\label{thm:EFMcD}
Any Weinstein filling of the standard contact sphere $(S^{2n-1},\xi_{\rm std})$ is diffeomorphic  to 
the ball $B^{2n}$.
\end{theorem}

Generalizing  Theorem \ref{thm:EFMcD}
K. Barth, H. Geiges and K. Zehmisch proved in \cite{BGZ16}:
\begin{theorem}\label{thm:BKZ}
All  Weinstein fillings of a simply connected  contact manifold admitting  a subcritical filling are diffeomorphic.
\end{theorem}

In fact,  both Theorems  \ref{thm:EFMcD} and \ref{thm:BKZ}  hold in a stronger form for a more general class of symplectic, and not necessarily Weinstein fillings.
We also note that  while it follows from Theorem \ref{thm:Weinstein-flexible} that all  completed subcritical  Weinstein fillings of a given contact manifold  are symplectomorphic (we note that the $(n-1)$-connectedness of the inclusion map $\p X\hookrightarrow X$ implies that  the homotopy class of   an almost complex  structure on a subcritical manifold is determined by the homotopy class of its restriction to the boundary), it is unknown for $n>2$  whether {\em all} completed  fillings of a contact manifold admitting a subcritical filling (e.g. the standard contact sphere) are symplectomorphic.

The following theorem of Oleg Lazarev constrains topology of {\em flexible} Weinstein manifolds.
\begin{theorem}[\cite{Laz-new}]\label{thm:Lazarev1}
All flexible fillings of of a contact manifold $(Y,\xi)$ with $c_1(Y,\xi)=0$ have canonically isomorphic integral homology.
\end{theorem}
In particular, as  Lazarev observed,  Theorem \ref{thm:Lazarev1} together   Smale's classification of $2$-connected $6$-manifolds from \cite{Smale}  and the fact that $\pi_3(O/U)=0$ yield  a complete classification of flexibly fillable contact structures on $S^5$.
\begin{corollary}[\cite{Laz-new}]
There exists a  sequence $\xi_n, n=0,1,\dots,$ of pairwise non-contactomorphic contact structures on $S^5$ such that
\begin{itemize}
\item any flexibly fillable contact structure on $S^5$ is contactomorphic to one of the structures from this sequence;
\item the contact structure $\xi_0$ is standard;
\item   for $n\geq 1$
the contact sphere  $(S^5,\xi_n)$ admits a unique up to symplectomorphism  flexible Weinstein filling 
  diffeomorphic to 
  $(\mathop{\#}\limits_1^n S^3\times S^3)\setminus B^6$.
  \end{itemize}
\end{corollary}

There are further constraints on the topology of flexible Weinstein fillings.   
In particular, 
\begin{thm}[\cite{EGL-signature}] Let $(S^{4n-1},\xi)$ be a flexibly fillable contact  structure. Then the signature of its flexible  filling is uniquely determined by the contact structure $\xi$.
\end{thm}

\begin{problem}
Does a contact structure $(Y,\xi)$ remember
\item{a)} the  diffeomorphism type of its flexible Weinstein filling $(X,\om,Z,\phi)$?
\item{b)} the  almost symplectic homotopy class $[\om]$ of the symplectic structure $\om$?
\end{problem}

We note that the diffeomorphism type of $X$ together with the homotopy class $[\om]$ determine a flexible Weinstein structure up to Weinstein homotopy, and hence  the positive answer to a) and b)
would imply that the contact structure $(Y,\xi)$ remember the symplectomorphism type of the completion of its flexible  filling.

\section{Nadler's  program of arborealization}
A priori, a  skeleton of a Weinstein domain can have very complicated singularities.
However, David Nadler conjectured that up to Weinstein homotopy the singularities of the skeleton can be reduced to a finite list in any dimension, see \cite{Nadler1}.
For $2n$-dimensional symplectic Weinstein manifolds the list of Nadler's singularities,   which he calls {\em arboreal},  are enumerated by decorated rooted trees with $\leq n+1$ vertices. It is remarkable that the singularity of each  given type has a unique symplectic realization.
Nadler also proposed in \cite{Nadler2} a procedure for arborealization of the skeleton of a Weinstein structure.
His procedure replaces a given Weinstein structure by another one whih an arboreal skeleton.
Nadler proved in \cite{Nadler2} that the constructed Weinstein  manifold has microlocal sheaf-theoretic invariants  equivalent to those of the Weinstein manifold. Conjecturally this  implies that the wrapped Fukaya categories are  also the same  for   the original and modified Weinstein manifold. However, it is unclear whether Nadler's  modification yields   a  Weinstein structure  which  is homotopic, or even symplectomorphic to the original one.  

In an    ongoing joint project \cite{ENS} with David Nadler and  Laura Starkston   we are exploring a somewhat different strategy  for arborealization of the Weinstein skeleton via a Weinstein homotopy  using simplification of singularities type technique in the spirit of a recent paper of D. \'{A}lvarez-Gavela, \cite{Dani}.  In some special cases this program was already carried out by   Starkston in \cite{Starkston}.

In this section we discuss the arboreal singularities with more detail and give   precise statements  of some of the  results from \cite{ENS}.

 \subsection{Definition of an arboreal singularity}

While we  define below   arboreal models as closed properly embedded subsets of the standard symplectic  vector space,   we are interested only in germs of these models at the origin.

Consider a tree $T$ with $\leq n+1$ vertices and a fixed vertex $R$, the {\em root}.  Suppose in addition that all edges, except the terminal ones  are decorated with $\pm1$. We will denote by $\eps$ the decoration, and   by $|T|$ the total number of   vertices.   With each decorated  rooted    tree $(T,\eps)$  we associate a unique up to symplectomorphism model $A(T,\eps,m) \subset \R^{2m}=T^*\R^m$ in each dimension  $m\geq n$ of the skeleton.  The models will be stratified by   strata which are isotropic for the Liouville form $pdq$.
 In dimension $m>n$  we have $A(T,\eps,m)=A(T,\eps,n)\times\R^{m-n}\subset T^*\R^n\times T^*\R^{m-n}= T^*\R^m$.

\begin{figure}[h]
\begin{center}
\includegraphics[scale=.7]{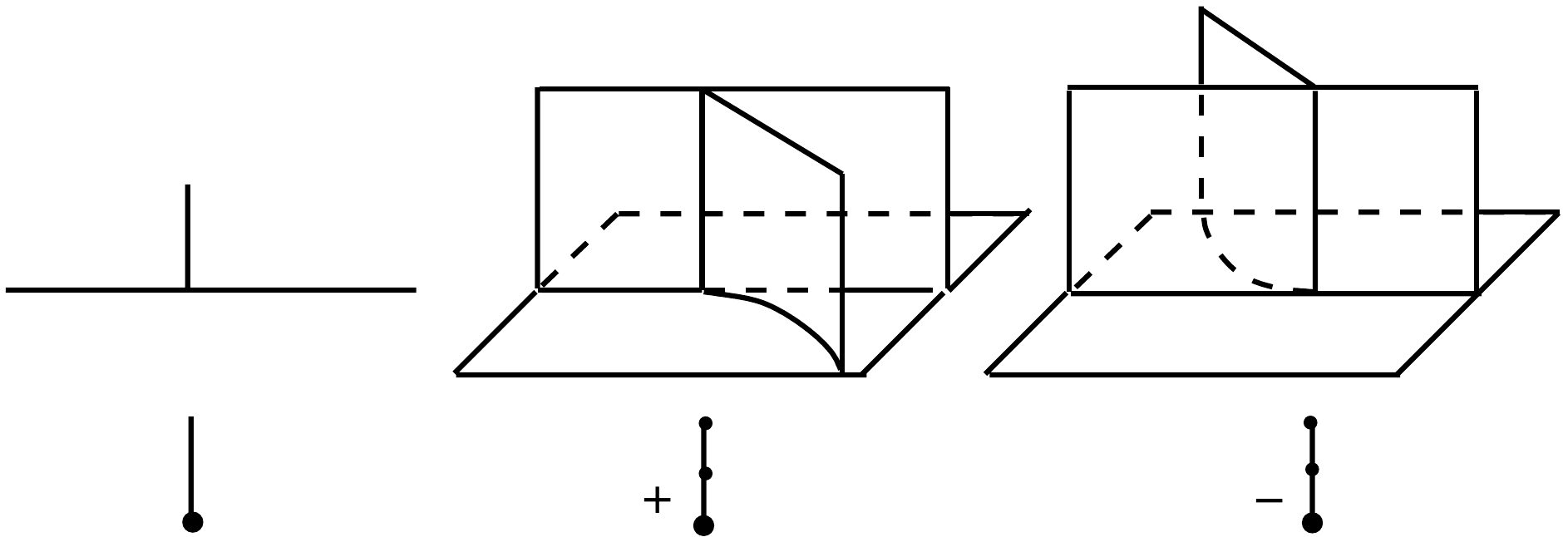}
\smallskip

\includegraphics[scale=.7]{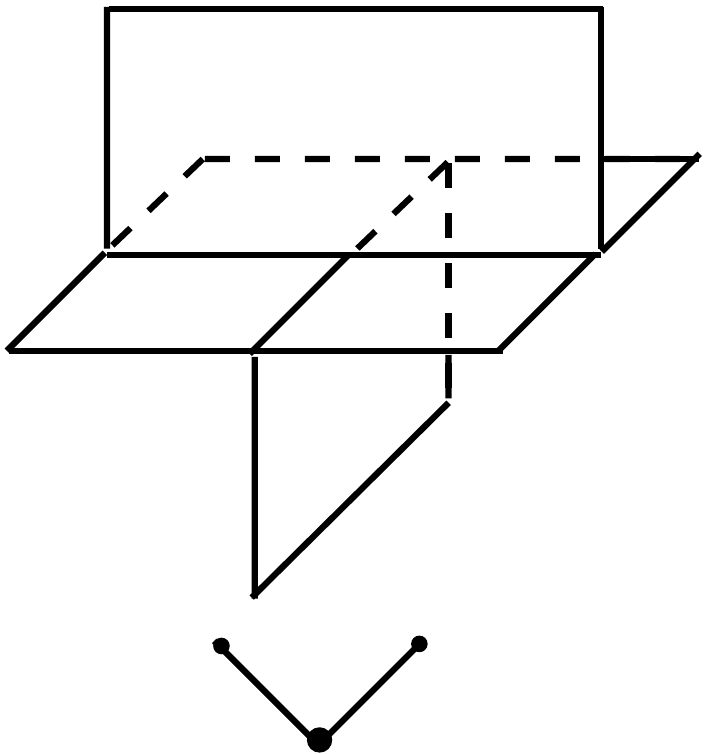}
\caption{Arboreal singularities labeled by rooted decorated trees. The picture represents Lagrangian skeleta themselves, and not their front projections. Free boundaries of vertical strata form Legendrian trees, while their traces at the horizontal plane are fronts of these trees.}\label{fig:arbor}
\end{center}  
\end{figure}

The model $A(T,\eps,n)$ will be defined inductively in  $n$.
For a tree $T$ which consists of one vertex we define $A(T,0)$ to be a point (in the $0$-dimensional symplectic space  $T^*\R^0$), and respectively
$A(T,n)=\R^{n}\subset T^*\R^n$.

As it was already stated above,   the Liouville form $pdq$ vanishes on each stratum of the    model $A(T,\eps,n)\subset T^*\R^n$.  Hence, if we view  $T^*\R^{2n}$ as a (Weinstein) hypersurface $\{z=0\}$ in the contact space $(\R^{2n+1}=T^*\R^n\times\R, pdq+dz)$, then  all strata of $A(T,\eps,n)\subset T^*\R^n$ are also isotropic  for the contact form $pdq+dz$. However, unless $A(T,\eps,n)$ is a Lagrangian plane, the front projection $(p,q,z)\mapsto (q,z)$ is very degenerate, because it collapses the image to the hyperplane $\{z=0\}$. We want to deform the model $A(T,\eps,n)$ in $\R^{2n+1}$ to make the front projection more generic.
To do that, consider a contactomorphism 
$S:\R^{2n+1}\to \R^{2n+1}$  given by the formula
$$S(p_1,\dots,p_n, q_1,q_2,\dots, q_n, z)=(p_1,\dots,p_n, q_1+p_1,q_2,\dots, q_n, z-\frac{p_1^2}2).$$ Then $$S^{-1}(p_1,\dots,p_n, q_1,q_2,\dots, q_n, z)=(p_1,\dots,p_n, q_1-p_1,q_2,\dots, q_n, z+\frac{p_1^2}2).$$ 
 Denote 
$$ \wh A^+(T,\eps,n):=S(A(T,\eps,n)),\;\; \wh A^-(T,\eps,n):=S^{-1}(A(T,\eps,n)) .$$
The sets $ \wh A^\pm(T,\eps,n)$ are stratified  by isotropic for the contact form $dz+pdq$ strata. If $|T|=1$ we have $ \wh A^+(T,n)= A(T,n)$.

Suppose that we already defined models  for all decorated rooted trees $(T,\eps)$ with $|T|\leq n $.
 Consider a rooted tree $(T,\eps)$ with $|T|=n+1$.
 By removing the root $R$  and all edges adjacent  to $R$ we get $k$ decorated trees 
 $(T_1,\eps), \dots, (T_k,\eps_k)$  with $|T_1|=n_1,\dots, |T_k|=n_k , \; n_1+\dots+n_k=n$.  For each of them we choose as its root the  vertex which was connected  in $T$ to $R$. Let $\sigma_j=\pm1$ be the decoration  of the edge  which was connecting   the root $R$ with  the root of the tree $T_j$, $j=1,\dots, k$.

 Consider   already defined models $ A(T_1,\eps_1,n-1),\dots, A(T_k,\eps_k,n-1)\subset T^*\R^{n-1}\times\R$.

Denote  $N_0:=0$, $N_j:=\sum\limits_{i=1}^j n_i ,\;\; j=1,\dots, k-1$.
For each $j=0,\dots, k-1$ 
 consider  the hyperplane $\Pi_j=\{p_{N_j+1}=1\}$  in   $\R^{2n}=T^*\R^{n }$ with the Liouville form $\lambda=\sum\limits_1^n p_j dq_j$. Note that $\Pi_j$ is transverse to the Liouville vector field $Z=\sum\limits_1^np_j\frac{\p}{\p p_j},$ or equivalently  $\lambda|_{\Pi_j}=dq_{N_j+1}+\sum\limits_{i\in\{1,\dots, n\}, i\neq N_j+1} p_idq_i$
 is a contact form.   Cyclically  ordering  coordinates $q_{N_j+2},\dots, q_n, q_1,\dots, q_{N_j}$  and taking the coordinate $q_{N_j+1}$ as $z$ we  identify  $ \Pi_j$ with $T^*\R^{n-1}\times \R$. Consider  $ A^{\sign(\sigma_j)}(T_j,\eps_j, n-1)\subset \Pi_j$.

  Denote $$B(T,\eps, n):=\{(tp,q)\in T^*\R^n;\; t\in[0,\infty), (p,q)\in \bigcup\limits_{j=1}^k  \wh A^{\sign(\sigma_j)}(T_j,\eps_j, n-1))\}.$$  Note that  $B(T,\eps, n) \cap\{p=0\}$ is the union of front projections of Legendrian complexes $ \wh A^{\sign(\sigma_j)}(T_j,\eps_j, n-1))$, and  $ B(T,\eps, n)$ is the positive conormal of this stratified set  co-oriented by the vector field $\frac{\p}{\p q_{N_j+1}}$.
Finally, we define  $$A(T,\eps, n):=\{p=0\}\cup B(T,\eps, n).$$ 
Singularities of the form  $A(T,\eps, n)$ where $(T,\eps)$ is a decorated rooted tree are called
{\em primary arboreal}.

 Note that up to   linear symplectomorphism the result of the above construction is independent of the ordering  of the trees $T_1,\dots, T_k$. Indeed, the corresponding symplectomorphism is the symplectization  of the linear automorphism of $\R^n$ appropriately permuting the coordinates $q_1,\dots, q_n$.
 
As an example, let us explicitly construct the models shown on Fig, \ref{fig:arbor}.
For a tree with 2 vertices we take the standard symplectic $\R^2$ with coordinates
$(p,q)$. Then $\Pi =\{p=1\}$.  For the $1$-vertex  tree $T_1$ the model $A(T_1,0)$ coincides is the point $ \{p=1,q=0\}\in\Pi$ and $\wh A(T_1,0)= A(T_1,0)$. Hence   $B(T,1)=\{(t,0), t\geq 0\}$ is the positive $p$-semi-axis, and 
 $A(T,1)=\{p=0\}\cup B(T,2),$ is the union of the coordinate line $q$ with this semi-axis, as it is shown on the left side of Fig.  \ref{fig:arbor}.
 
 For the rooted tree with three vertices and the central root, as on the lower picture in Fig. \ref{fig:arbor}, each of   the trees  $T_1, T_2$ has 1 vertex.
 Hence,  $\Pi_1=\{p_1=1\}, \Pi_2= \{p_2=1\}$, and identifying this hyperplanes with the standard contact $\R^3$ we get
 $A(T_1,1)=\{p_2=q_1=0\}\subset\Pi_1$ and $A(T_2,1)=\{p_1=q_2=0\}\subset\Pi_2$. Therefore,
 $$A(T,2)=\{p=0\}\cup\{p_2=q_1=0,p_1\geq0\}\cup\{p_1=q_2=0,p_2\geq 0\}.$$

  Finally, consider  the  right models on  Fig.  \ref{fig:arbor}.
 The models are contained in the standard symplectic $\R^4$ with canonical coordinates $(p_1,q_1,p_2,q_2)$, and we have  $\Pi=\Pi_1=\{p_1=1\}$  The tree $T_1$ in this case consists of two vertices, and  identifying   $\Pi$ with the standard symplectic $\R^2$, we find that $$ \wh A^\pm(T_1,1) =\{q_1=p_2=0\}\cup\{q_1= \mp p_2^2, p_2=\pm q_2, p_2\geq 0\}.$$ Note that the second stratum in the union can also be written as $\{q_1= \mp q_2^2, p_2=\pm q_2, p_2\geq 0\}$
  Thus we have 
\begin{align*}
&B(T, +1, 2)=\{p_2=0, q_1=0, \;p_1\geq 0\}\cup \{q_1= - 
\frac {q_2^2}2, p_2= p_1q_2, \;p_1,p_2\geq 0\},\\
&B(T, -1, 2)=\{p_2=0, q_1=0, \;p_1\geq 0\}\cup \{q_1=  
\frac {q_2^2}2, p_2= -p_1q_2, \;p_1,p_2\geq 0\}
\end{align*}
Note that $B(T,\pm1,2)\cap\{p=0\}=\{q_2=0\}\cup\{q_1=\mp\frac{q_2^2}2\}$ is the front of the Legendrian tree $\wh A^\pm(T_1,1)$, while $B(T,\pm1,2)$ is the positive conormal of this front co-oriented by the vector field $\frac{\p}{\p q_1}$.  \medskip

A {\em general  } arboreal singularity is associated to a 
{\em double decorated }
rooted  tree  with  an additional decoration $\beta$ which assigns $0$ or $1$ to  all  terminal vertices of the tree $T$. We extend $\beta$ to all vertices  by setting $\beta(v)=0$ for all non-terminal vertices.  Primary arboreal  singularities correspond to the case when the decoration $\beta$ is identically $0$.

 We denote  $|\beta|:=\sum\beta(v)$, where the sum is taken over all terminal vertices $v$ of the tree $T$.
  With each double decorated tree $(T,\eps, \beta)$ we associate a unique up to symplectomorphism model
$A(T,\eps,\beta, m)\subset T^*\R^m$ for each $m\geq |T|+|\beta|-1$.
    In dimension $m \geq n:= |T|+|\beta|-1 $  we have $$A(T,\eps,\beta, m)=A(T,\eps,\beta, n)\times\R^{m-n}\subset T^*\R^n\times T^*\R^{m-n}= T^*\R^m.$$
    
    The model  $A(T,\eps,\beta, m)\subset T^*\R^m$  with $m= |T|+ |\beta| -1$ is defined by a similar inductive   procedure  as for primary arboreal singularities, beginning  with  $$A(T,\eps,\beta,1)=\{p=0,q\geq 0\}\subset T^*\R \;\;\hbox{
   for}\;\;  |T|=1\;\; \hbox{and } \; |\beta|=1.$$
  
  Every model  $A(T,\eps,\beta, m)\subset T^*\R^m$ can be presented as a union of Lagrangian sheets
  $L_v$ enumerated by vertices of the graph $T$. Denote by $d(v)$ the distance between $v$ and the root. Then  $L_v$ is diffeomorphic to
  the quadrant $$\{(x_1,\dots, x_n)\in\R^n; x_1,\dots x_k\geq0\},\quad  k=d(v)+\beta(v).$$
  
   Note that  the model  $A(T,\eps,\beta, m)$ inherits a smooth structure (i.e. the algebra of smooth functions) from the ambient space $\R^{2n}$. 
   By an $n$-dimensional {\em arboreal complex} we mean a set covered by charts diffeomorphic to one of the models $A(T,\eps,\beta, n)$. Hence, every arboreal complex can be canonically stratified by  strata $S_{T,\eps,\beta}$ of dimension $n-|T|-|\beta|+1$. A diffeomorphism  $f:C\to C'$ between two arboreal complexes induces a diffeomorphism between the corresponding strata, but not every continuous map $f:C\to C'$ which is a diffeomorphism on the corresponding strata is a diffeomorphism of arboreal complexes $C$ and $C'$.

\subsection{Main results}  
\begin{prop}[\cite{ENS}]
For each  arboreal complex $C$ there exists a unique up to symplectomorphism   Weinstein domain $ \fW(C)=(X,\om, Z,\phi)$,   {\em ``the cotangent bundle"} of $C$ such that $C=\Core(X,\om, Z)$. Any two  such Weinstein structures  $(X,\om, Z,\phi)$ and $(X,\om, Z',\phi')$ are homotopic through  a family of Weinstein structures with a fixed core.
\end{prop}  
  \begin{thm}[\cite{ENS}] \label{thm:arboreal-existence}
  \begin{enumerate}
  \item Any Weinstein structure is homotopic to a Weinstein structure with an arboreal skeleton.  
  \item Let $\fW_t$, $t\in[0,1]$ be a Weinstein homotopy such $\fW_0$  and $\fW_1$ have arboreal skeleta. Then there exists a Weinstein pair structure $(\fW;\fW_0\cup\fW_1)$ on $X\times T^*I$ with an arboreal skeleton which is homotopic to the Weinstein pair associated to the homotopy $\fW_t$ {\rm (see Section \ref{sec:hom-pair})}.
  \end{enumerate}
    \end{thm}
  
  Under some topological constraints on the manifold $X$ one can further restrict the list of necessary singularities.
  \begin{thm}[\cite{ENS}]\label{thm:arboreal-complexification}
  Let $\fW=(X,\om, Z,\phi)$ be a Weinstein structure. Suppose that
  \begin{description}
  \item{a)} the manifold $X$ is $(n-2)$-connected;
  \item{b)} there exists a field of Lagrangian planes $\tau\subset TX$; in other words, $TX$ with its homotopically canonical almost complex structure is isomorphic to the complexification of a real  $n$-dimensional  vector bundle.
  \end{description}
  Then the Weinstein structure $\fW$ is homotopic to a Weinstein structure
  $\wt\fW=(X,\om, \wt Z,\wt\phi)$ whose skeleton is an arboreal complex with singularities of type
  $(T,\eps,\beta)$ where the distance from the root of the tree $T$ to any other  vertex is no more than $2$ and the decoration $\eps$ takes only positive values.
  \end{thm}

\end{document}